\documentclass[12pt]{article}
\usepackage{amsmath, amsfonts, amsthm, amssymb}
\usepackage{graphicx}
\usepackage{float}
\usepackage{verbatim}

\numberwithin{equation}{section}

\newtheorem{thm}{Theorem}[section]
\newtheorem{lma}[thm]{Lemma}

\newtheorem{defn}[thm]{Definition}

\newtheorem{prop}[thm]{Proposition}

\newtheorem{ques}[thm]{Question}

\renewcommand{\ge}{\geqslant}
\renewcommand{\le}{\leqslant}
\renewcommand{\geq}{\geqslant}
\renewcommand{\leq}{\leqslant}
\renewcommand{\H}{\text{H}}

\allowdisplaybreaks

\title{Assouad type dimensions for self-affine sponges}

\author{Jonathan M. Fraser \& Douglas C. Howroyd\\ \\
\emph{School of Mathematics,}\\ \emph{ Alan Turing Building,}\\ \emph{University of Manchester,}\\ \emph{ Manchester M13 9PL, UK}\\ \\
\emph{E-mail contact}:\\ jon.fraser32@gmail.com\\  douglas.howroyd@gmail.com }

\begin{document}

\maketitle

\begin{abstract}
We study the Assouad and lower dimensions of self-affine sponges; the higher dimensional analogue of the planar self-affine carpets of Bedford and McMullen.  Our techniques involve the weak tangents of Mackay and Tyson as well as regularity properties of doubling measures in the context studied by Bylund and Gudayol. \\


\end{abstract}

\section{Introduction} \label{intro}

The Assouad and lower dimensions are increasingly popular notions of dimension used to study the irregularity of fractal sets.  In this paper we compute these dimensions for a certain class of self-affine sets in $\mathbb{R}^d$ ($d\in \mathbb{N}$), which we refer to as \emph{self-affine sponges}, following Kenyon and Peres \cite{kenyonperes}. These are the natural higher dimensional analogue of the planar self-affine carpets of Bedford \cite{bedford} and McMullen \cite{mcmullen}, which have been extensively studied since their inception in the mid 1980s.  The Assouad dimension of the planar \emph{carpets} was first computed by Mackay \cite{mackay} and the lower dimension was computed later by Fraser \cite{fraser}.  Our work can be seen as a natural extension of these papers to the higher dimensional case.  We wish to point out at this stage that this extension is not straightforward and we have to introduce several new ideas in our proofs.  The dimension formulae themselves are also not immediate generalisations of the formulae in the planar case due to the fact that `dimension' has to be maximised in each of the last $(d-1)$ coordinates independently.  This will become clear upon inspection and comparison of our results.

There are two main techniques we employ.  The first is used to bound the Assouad dimension from above and the lower dimension from below and is based on a useful measure theoretic formulation of these two dimensions due to Konyagin and Vol'berg \cite{konyagin} and Bylund and Gudayol \cite{bylund}.  This replaces the delicate covering arguments used by Mackay and Fraser \cite{mackay, fraser} and seems to us more appropriate to handle the higher dimensional case. For this part of the proof we will somewhat rely on Olsen's treatment of a natural family of measures supported on sponges \cite{sponges}.  The second technique is used to bound the lower dimension from above and the Assouad dimension from below and is based on finding appropriate \emph{weak tangents} to our sets. This approach was pioneered by Mackay and Tyson \cite{mackaytyson} and used by Mackay and Fraser in the two dimensional case \cite{mackay, fraser}.

The measure theoretic formulation of the Assouad and lower dimensions asks for measures with certain scaling properties and leads to the natural question of when there are `sharp' measures, i.e., particular measures with the precise scaling property, rather than a collection of measures approximating the scaling property.  We show that such sharp measures do exist (and can be taken to be Bernoulli) if the \emph{very strong separation condition} of Olsen \cite{sponges} is satisfied.  We also point out that there are sponges in our class without this condition for which \emph{all} Bernoulli measures fail to be doubling, and hence cannot be sharp. This follows from \cite{doublingcarpets}, but we provide a simple example for completeness.  

Also in our Examples section (Section \ref{qanda}), we  provide a smooth 1-parameter family of self-affine carpets for which the Assouad and lower dimensions exhibit a new kind of discontinuity and show that for sponges in our class either the Assouad, box, Hausdorff and lower dimensions are all equal or all distinct, extending a dichotomy of Mackay \cite{mackay} to the higher dimensional setting.

\subsection{Assouad dimension and lower dimension} \label{dimension}

Throughout this section let $F \subseteq \mathbb{R}^d$ be a non-empty compact bounded set.  The Assouad dimension of $F$ is defined by 
\begin{multline*}
\dim_{\text{A}} F = \inf \Bigg\{ s \geq 0 \, \,  : \, \exists \text{ constants } C,\rho > 0 \text{ such that, for all } \, \, 0< r< R \leq \rho,\\ \text{ we have  }\sup_{x\in F} N_r (B(x,R)\cap F) \leq C\left(\frac{R}{r}\right)^{s} \Bigg\}\end{multline*}
where $B(x,R)$ means open ball of radius $R$ and centre $x$, $N_r(E)$ is the smallest number of open sets in $\mathbb{R}^d$ with diameter less than or equal to $r$ required to cover a bounded set $E$. We are also interested in the following measure theoretic formulation of the Assouad dimension due to \cite{luksak, konyagin}: \newpage
\begin{multline*} 
\dim_{\text{A}} F = \inf \Bigg\{ s \geq 0 \, \,  : \, \exists \text{ a Borel probability measure } \mu \text{ fully supported by }F \\  \text{ and constants }C,\rho > 0\text{  such that,} \\ \text{for all } \, \, 0< r< R \leq \rho, \text{ we have  }  \sup_{x\in F} \frac{\mu(B(x,R))}{\mu(B(x,r))} \leq C\left(\frac{R}{r}\right)^{s} \Bigg\}.
\end{multline*}

Assouad dimension has a natural dual that we will call the lower dimension $\dim_\text{L} F$, following Bylund and Gudayol \cite{bylund}. This notion was first introduced by Larman \cite{larman}, where it was refereed to as the \emph{minimal dimensional number}.  We will have two definitions for the lower dimension, both duals of the respective definitions for the Assouad dimension, which emphasises the link between the two dimensions. First, the definition involving covers is
\begin{multline*}
\dim_{\text{L}} F = \sup \Bigg\{ s \geq 0 \, \,  : \, \exists \text{ constants }C, \rho > 0 \text{ such that, for all } \, \, 0< r< R \leq \rho,\\ \text{ we have  }\inf_{x\in F} N_r (B(x,R)\cap F) \geq C\left(\frac{R}{r}\right)^{s} \Bigg\}.
\end{multline*}

Bylund and Gudayol \cite{bylund} proved a result linking doubling measures and lower dimension, similar to \cite{luksak, konyagin} (see below for definition of doubling measure):
\begin{multline*}
\dim_{\text{L}} F = \sup \Bigg\{ s \geq 0 \, \,  : \, \exists \text{ a doubling Borel probability measure } \mu \text{ fully supported on }F \\ \text{ and constants }C, \rho > 0  \text{ such that, } \\ \text{for all } \, \, 0< r< R \leq \rho, \text{ we have  }\inf_{x\in F} \frac{\mu(B(x,R))}{\mu(B(x,r))} \geq C\left(\frac{R}{r}\right)^{s} \Bigg\}.
\end{multline*}

For a more in-depth discussion of the Assouad and lower dimension and for some of their basic properties, we refer the reader to \cite{robinson, luk, fraser}. Assouad and lower dimension provide coarse and extremal information about the geometric structure of a set.  In particular, they identify the parts of the set which are `thickest' and `thinnest' respectively and unsurprisingly provide upper and lower bounds for other commonly used notions of dimension, which describe more of an `average thickness'.  We will often refer to Hausdorff and upper and lower box dimensions, which we denote by $\dim_\H$,  $\overline{\dim}_\text{B}$ and $\underline{\dim}_\text{B}$ respectively, and refer the reader to \cite[Chapters 2--3]{falconer} for their definitions and basic properties.  In general, for a compact set $F$, we have 
\[
\dim_\text{L} F  \ \leq \  \dim_\H F\ \leq \  \underline{\dim}_\text{B} F  \ \leq \  \overline{\dim}_\text{B} F \ \leq \  \dim_\text{A} F,
\]
where the first inequality is due to Larman \cite[Theorem 5]{larman}.  If the upper and lower box dimensions coincide, we will denote the common value by $\dim_\text{B} F$ and refer to it as the box dimension.  This will be the case for the self-affine sets considered in this paper. Of course, the measure theoretic formulations of the Assouad and lower dimensions do not guarantee the existence of `sharp' measures, i.e., measures satisfying the required scaling property with $s$ actually \emph{equal} to $\dim_{\text{A}} F$ or $\dim_{\text{L}} F$.  We are interested in the general question of which sets carry sharp measures and what form they take.  This question was mentioned explicitly in the context of self-affine sets in \cite[Question 4.4]{fraser}.  We provide a partial solution to this question here.

Note that any measures $\mu$ used in the measure theoretic definition of Assouad dimension above must be doubling.  Recall that a measure is \emph{doubling} if there exists a constant $C>0$ such that for all $x$ in the support of $\mu$ and all $r>0$ we have
\[
\frac{\mu(B(x,2r))}{\mu(B(x,r))}  \ \leq \ C.
\]
It is well-known that if a measure is doubling, then for any $c>1$, there exists a constant $C'>0$ such that for all $x$ in the support of $\mu$ and all $r>0$ we have
\[
\frac{\mu(B(x,cr))}{\mu(B(x,r))}  \ \leq \ C'.
\]

\subsection{Self-affine sponges} \label{sponges}

Self-affine carpets have been investigated in depth over the last 30 years, partially because they provide good examples of self-affine sets that can be visualised easily and studied explicitly. Roughly speaking, a \emph{carpet} refers to a self-affine set in the plane which is the attractor of an iterated function system consisting of affine maps corresponding to diagonal matrices (or at least matrices which map the coordinate axes onto themselves).  There exist several classes of self-affine carpets of which the Bedford--McMullen class are the oldest and simplest. Specific attention has been paid to dimension theoretic properties and we now have formulae for the Hausdorff, box, packing, Assouad and lower dimensions in certain instances, see \cite{baranski, bedford, fengaffine, lalley-gatz, mcmullen, fraser_box, fraser, mackay}. Kenyon and Peres \cite{kenyonperes} first studied the higher dimensional analogue of the planar Bedford--McMullen carpets, which they called \emph{sponges}, and calculated the Hausdorff and box dimensions.  Olsen \cite{sponges}  considered the same model and studied the multifractal structure of the corresponding Bernoulli measures. In general, much less is known in the higher dimensional setting of sponges.

We follow the notation set up by Olsen \cite{sponges}.  The \emph{Bedford--McMullen sponges} are defined as follows. Let $d\in \mathbb{N}$ and, for all $l=1, \ldots, d$, choose $n_l \in \mathbb{N}$ such that $1 < n_1 < n_2 < \cdots < n_d$.  Note that the construction is still valid if some of the inequalities between the $n_l$ are actually equalities, but surprisingly this causes some problems in our proofs, see Section \ref{examplenonstrict}.   Let $\mathcal{I}_l=\left\{0,\ldots, n_l-1 \right\}$ and $\mathcal{I}=\prod_{l=1}^{d}\mathcal{I}_l$ and consider a fixed digit set $D \subseteq \mathcal{I}$ with at least two elements.  For $\textbf{i} =\left( i_1,\ldots, i_d\right)\in D$ we define the affine contraction $S_\textbf{i}\colon [0,1]^d \rightarrow [0,1]^d$ by
\[
S_{\textbf{i}}\left(x_1, \ldots, x_d \right) = \left( \frac{x_1+i_1}{n_1},\ldots,\frac{x_d+i_d}{n_d} \right) .
\]
Using a theorem of Hutchinson $\cite{hutch}$, there exists a unique non-empty compact set $K \subseteq [0,1]^d$ satisfying 
\[
K=\bigcup_{\textbf{i}\in D}S_{\textbf{i}}(K). 
\]
called the attractor of the iterated function system (IFS) $\ \left\{ S_{\textbf{i}} \right\}_{\textbf{i}\in D}$.  The set $K$ is the \emph{self-affine sponge} and our main object of study. It is worth noting that if all our $n_l$ are equal then our set is self-similar satisfying the open set condition and the Assouad dimension equals the similarity dimension, see for example \cite{olsenassouad}, and if $d=2$ then we are in the self-affine carpet setting of Bedford-McMullen. We will assume without loss of generality that $K$ does not lie in a hyperplane.  If this was the case, we could restrict our attention to the minimal lower hyperplane containing $K$ and consider it as a self-affine sponge in this space.

We will often model our sponge $K$ via the symbolic space $D^{\mathbb{N}}$, which consists of all infinite words over $D$ and is equipped with the product topology generated by the cylinders corresponding to finite words over $D$.  We define a continuous (but not necessarily injective) map $\tau :D^{\mathbb{N}} \rightarrow [0,1]^d$ by
\[
\{\tau(\omega)\}=\bigcap_{n \in \mathbb{N}} S_{\omega\vert n}([0,1]^d)
\]
where $\omega= (\textbf{i}_1, \textbf{i}_2,\ldots)$, $\omega \vert n=\left( \textbf{i}_1, \ldots , \textbf{i}_n \right) \in D^n$ and $\textbf{i}_j=(i_{j,1},\ldots,i_{j,d})$ for any $j\in \mathbb{N}$, and $S_{\omega\vert n} = S_{\left( \textbf{i}_1,  \ldots ,\textbf{i}_n  \right) } =  S_{ \textbf{i}_1} \circ \cdots  \circ S_{ \textbf{i}_n} $.

This allows us to switch between symbolic notation and geometric notation since
\[
\tau(D^{\mathbb{N}})=K.
\]
\emph{Pre-fractals} of an attractor $F$ (for a given IFS) are sets that are defined by the application of all possible combinations of functions in our IFS to an initial set a certain number of times (this number will determine the level of the pre-fractal). As the level tends to infinity, the pre-fractals will converge to our attractor in the Hausdorff metric, regardless of our initial set; for more details see \cite[page 126]{falconer}. For us, the $n$th pre-fractal in the construction of $K$ is
\[
\bigcup_{\left( \textbf{i}_1,  \ldots ,\textbf{i}_n  \right)\in D^{n} } S_{\left( \textbf{i}_1,  \ldots ,\textbf{i}_n  \right) }([0,1]^d),
\]
which is a collection of $\lvert D \rvert^n$ hypercuboids, and these help to visualise the set $K$ itself.  

\begin{figure}[htbp]
	\centering
    \includegraphics[width=150mm]{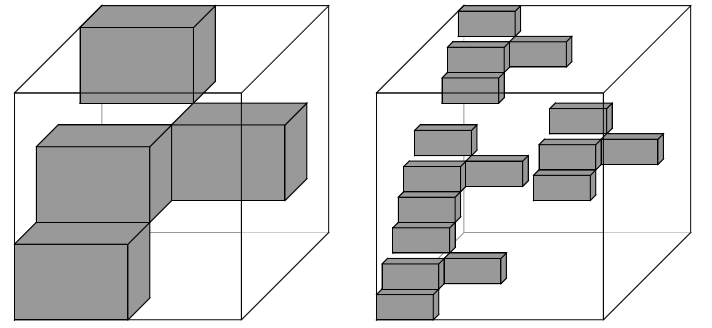}
	\caption[Simple Sponge]{The first and second level in the construction of a particular self-affine sponge in $\mathbb{R}^3$ where $n_1=2$, $n_2=3$, $n_3=4$ and $D = \{(0,0,0), (0,1,1), (0,2,3), (1,1,2) \}$.}
	\label{fig:Self-affine sponge}
\end{figure}

A commonly considered and important class of Borel measures supported on fractal attractors are \emph{Bernoulli measures}.  Associate a probability vector $\ \left\{ p_{\textbf{i}} \right\}_{\textbf{i}\in D}$ with $D$ and let $\tilde{\mu}=\prod_{\mathbb{N}}\left( \sum_{\textbf{i}\in D}p_{\textbf{i}}\delta_{\textbf{i}} \right)$  be the natural Borel product probability measure on $D^{\mathbb{N}}$, where $\delta_{\textbf{i}}$ is the Dirac measure on $D$ concentrated at $\textbf{i}$.  Finally, the measure 
\[
\mu(A)=\tilde{\mu}\circ \tau^{-1}(A)
\]
for a Borel set $A\subseteq K$, is our Bernoulli measure supported on $K$.  We will be interested in one particular Bernoulli measure, which we now describe.  For any $l=1,2,\ldots,d$ we define $\pi_l \, : \, D \rightarrow \prod_{k=1}^l \mathcal{I}_k$ to be the projection onto the first $l$ coordinates, i.e., $\pi_l(i_1,\ldots,i_d)=(i_1,\ldots, i_l)$ and let $D_l=\pi_l (D)$ and $N=\#\{\pi_1 D\}$. Finally, for $l=1, \dots, d-1$ and $(i_1, \dots , i_l) \in D_l$ let
\[
N(i_1, \ldots, i_l)= \# \{ i_{l+1} \in \mathcal{I}_{l+1} : (i_1, \ldots, i_l, i_{l+1}) \in D_{l+1}\}
\]
 be the number of possible ways to choose the next digit of $(i_1, \dots , i_l)$. In particular, $N(i_1, \ldots, i_l)$ is an integer between 1 and $n_{l+1}$, inclusive.  The special Bernoulli measure we will use in this paper is that associated with the probabilities
\[
p_\textbf{i}=p_{(i_1,\ldots,i_d)}=\frac{1}{N\left(\prod_{l=2}^{d} N(i_1,\ldots,i_{l-1})\right)}.
\]
One can check from the definitions that $\sum_{ \textbf{i} \in D} p_\textbf{i} = 1$. We will refer to this measure as the `coordinate uniform measure', since its key feature is that it is defined inductively on the coordinates to be as uniform as possible.  We emphasise that this is \emph{not} the uniform measure given by $p_\textbf{i} = 1/\lvert D \rvert$.

\begin{figure}[H]
	\centering
		\includegraphics[width=80mm]{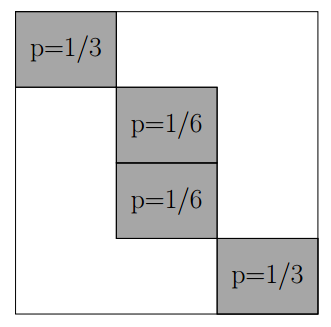}
	\caption[Example of probability measure]{An example of the coordinate uniform measure on the first level of a Bedford-McMullen carpet where $n_1=3$, $n_2=4$ and $D = \{(0,3), (1,1), (1,2), (2,0)\}$.}
	\label{fig:proba measure}
\end{figure}

\section{Results} \label{results}

We are now ready to state our main results, the first of which gives a simple and explicit formula for the Assouad and lower dimensions of a Bedford--McMullen sponge.

\begin{thm} \label{mainthm}
The Assouad dimension of $K$ is 
\[
\dim_{\text{\emph{A}}} K \ = \ \frac{\log N}{\log n_1} \ + \ \sum_{l=2}^d \frac{\displaystyle\log\max_{(i_1,\ldots, i_{l-1})\in D_{l-1} } N(i_1, \ldots, i_{l-1})}{\log n_l}
\]
and the lower dimension of $K$ is
\[
\dim_{\text{\emph{L}}} K \ = \ \frac{\log N}{\log n_1} \ +\ \ \sum_{l=2}^d \frac{\displaystyle\log\min_{(i_1,\ldots, i_{l-1})\in D_{l-1} } N(i_1, \ldots, i_{l-1})}{\log n_l}.
\]
\end{thm}

When $d=2$ we obtain the same formulae as Mackay and Fraser \cite{mackay, fraser} but there is an interesting difference between the carpet and the sponge. In Mackay's formula he considers one specific $\textbf{i}\in D$ which maximises the formula whereas for the sponge we consider multiple, potentially different $\textbf{i}\in D$, which can be thought of as each maximising the dimension in one coordinate. When first considering this question it is not obvious whether the Assouad dimension is obtained for one $\textbf{i} \in D$ (i.e., the maximum would be outside the sum) or for multiple $\textbf{i}\in D$ (maximum inside the sum). We find the second case to be true. This difference will be important in the proofs, notably the lower bound for Assouad dimension and the upper bound for lower dimension have additional complications since they will have to take into account all the different $\textbf{i} \in D$ used in the formulae.

Checking the subsequent proofs, one sees that the upper bound for Assouad dimension and lower bound for lower dimension remain valid when one allows $1 < n_1 \leq n_2 \leq \cdots \leq n_d$ with some of the inequalities replaced by equalities.  However, the arguments giving the lower bound for Assouad dimension and upper bound for lower dimension fail.  We will give an example in Section \ref{examplenonstrict} demonstrating that in fact our formulae do not generally remain valid if the $n_l$ are not strictly increasing.  This is perhaps surprising since the formulae for the Hausdorff and box dimensions are not sensitive to such a perturbation, see \cite[Corollary 4.1.2]{sponges} and note that Olsen does not assume that the $n_l$ are strictly increasing.

We now turn to the question of the existence of sharp measures supported by $K$.  For this we need Olsen's \emph{very strong separation condition} (VSSC), see \cite[4., condition (II)]{sponges}.
\begin{defn}[VSSC]
\emph{A sponge }$K$\emph{ (associated to }$D$)\emph{ satisfies the }very strong separation condition (VSSC)\emph{ if the following holds. If }$l=1, \ldots, d$\emph{ and }$(i_1, \ldots, i_d)$, $(j_1 ,\ldots, j_d)\in D$\emph{ satisfy }$i_1=j_1, \ldots,  i_{l-1}=j_{l-1}$\emph{ and }$i_l \neq j_l$,\emph{ then }$\lvert i_l - j_l \rvert >1$.
\end{defn}

The VSSC permits the following partial answer to \cite[Question 4.4]{fraser}.

\newpage

\begin{thm} \label{measurethm}
Let $K$ be a Bedford-McMullen sponge satisfying the VSSC and let $\mu$ be the coordinate uniform measure. Then there exist constants $C_0, C_1>0$ such that for all $0<r<R \leq 1$ and all $x \in K$, we have
\[
C_0 \, \left( \frac{R}{r} \right)^{\dim_{\text{\emph{L}}} K} \ \leq \ \frac{\mu\left(B(x,R)\right)}{\mu\left(B(x,r)\right)} \ \leq \ C_1 \, \left( \frac{R}{r} \right)^{\dim_{\text{\emph{A}}} K}.
\]
In particular, $\mu$ is a sharp measure for the Assouad dimension and lower dimension simultaneously.
\end{thm} 

The general question of whether there exist sharp measures for self-affine sponges and carpets without assuming the VSSC remains open and of interest to us.  In Section \ref{examplemeasures} we point out that Bernoulli measures cannot generally be used to solve this problem: specifically, there are self-affine carpets for which all Bernoulli measures fail to be doubling.  This was first proved by Li, Wei and Wen \cite{doublingcarpets}.  There the authors perform a detailed analysis of the precise conditions under which a Bedford-McMullen carpet carries a doubling self-affine measure, but they do not consider the existence of `sharp measures' in the context of Assouad dimension.  The existence of carpets which carry no doubling Bernoulli measures is perhaps surprising in light of the fact that self-similar sets satisfying the open set condition always carry a doubling Bernoulli measure which is simultaneously sharp for both lower and Assouad dimension (in fact it is Ahlfors regular).  More precisely, for a self-similar set defined via mappings with similarity ratios $\{c_i\}_{i \in \mathcal{J}}$ satisfying the open set condition and having Hausdorff dimension $s$, the Bernoulli measure corresponding to the probability vector $\{c_i^s\}_{i \in \mathcal{J}}$ satisfies
\[
C_0\left(\frac{R}{r}\right)^s \ \leq \ \frac{\mu(B(x,R))}{\mu(B(x,r))} \ \leq \ C_1\left(\frac{R}{r}\right)^s
\] 
for some uniform constants $C_0, C_1>0$ for all $x$ in the self-similar set and all $0<r<R \leq 1$.  For more information on self-similar sets, see \cite{hutch, falconer}.

\section{Proofs} \label{proof}

In this section we prove our main results, Theorems \ref{mainthm} and \ref{measurethm}.  Section \ref{intronotation} will introduce some important notation and concepts which will be used throughout the subsequent sections including \emph{approximate cubes}.  Theorem \ref{mainthm} will then be proved via four bounds given in Sections \ref{upperbound} through \ref{upperlower} and Theorem \ref{measurethm} will be proved in Section \ref{measurethmproof}.

\subsection{Important notation} \label{intronotation}

The symbolic-to-geometric projection map $\tau$ induces the coordinate functions $\tau_l$ which are just taken to be the $l^{th}$ coordinate of $\tau$.  For the application of these, note that the $S_{\textbf{i}}$, and therefore $\tau$, act independently on each coordinate.

We also let $\sigma:D^{\mathbb{N}} \rightarrow D^{\mathbb{N}}$ be the shift map $\sigma(\textbf{i}_1,\textbf{i}_2,\ldots)=(\textbf{i}_2, \textbf{i}_3, \ldots)$, which will be used to `zoom in' on certain interesting parts of our sponge.

Approximate cubes are well-known tools used in the study of self-affine sponges and will be used extensively throughout our proofs.  For all $r\in (0,1]$ we choose the unique integers $k_1(r),\ldots,k_d(r)$, greater than or equal to 0, satisfying
\[
\frac{1}{n_l^{k_l(r)+1}}< r \leq \frac{1}{n_l^{k_l(r)}}
\]
for $l=1,\ldots,d$.  In particular, 
\[
 \frac{-\log r}{\log n_l}-1 < k_l(r) \leq \frac{-\log r}{\log n_l}.
\]
Then the approximate cube $Q(\omega, r)$ of (approximate) side length $r$ determined by $\omega =\left( \textbf{i}_1, \textbf{i}_2 , \ldots \right) =\left( (i_{1,1}, \dots, i_{1,d}), (i_{2,1}, \dots, i_{2,d}) , \ldots \right)    \in D^{\mathbb{N}}$ is defined by
\[
Q(\omega, r)=\left\{ \omega'=\left( \textbf{j}_1, \textbf{j}_2 , \ldots \right)\in D^{\mathbb{N}} : \forall \, \,  l=1, \ldots, d \text{ and } \forall\, \, t= 1, \ldots, k_l(r) \text{ we have } j_{t,l}=i_{t,l} \right\}.
\]
Here our approximate cube is defined symbolically, which we find simplifies the proofs. The geometric analogue is $\tau\left(Q(\omega, r)\right)$, which is contained in
\[ 
\prod_{l=1}^d \left[\frac{i_{1,l}}{n_l}+\cdots+\frac{i_{k_l(r),l}}{n_l^{k_l(r)}} \, , \, \frac{i_{1,l}}{n_l}+\cdots+\frac{i_{k_l(r),l}}{n_l^{k_l(r)}}+\frac{1}{n_l^{k_l(r)}} \right];
\]
 a hypercuboid in $\mathbb{R}^d$ aligned with the coordinate axes and of side lengths $n_l^{-k_l(r)}$, which are all comparable to $r$ since $ r \leq n_l^{-k_l(r)} < n_l r$. This is why we call $Q(\omega,r)$ an approximate cube of side length $r$.

Finally, we will need to consider tangents to our sponge $K$ and these are defined using the Hausdorff metric $d_\mathcal{H}$ on the space of non-empty compact subsets of $\mathbb{R}^d$, which is defined by
\[
d_\mathcal{H}(A,B) \ = \  \inf \big\{ \varepsilon \geq 0   :   A \subseteq [B]_{\varepsilon} \text{ and } B \subseteq [A]_{\varepsilon} \big\}
\]
where $[A]_{\varepsilon}$ is the closed $\varepsilon$-neighbourhood of a set $A$.

\subsection{Upper bound for Assouad dimension} \label{upperbound}

For the upper bound we will use the measure theoretic definition of Assouad dimension.  Let $\mu$ be the coordinate uniform measure and $\{p_\textbf{i}\}_{\textbf{i} \in D}$ be the associated probabilities.  Olsen \cite[Section 3.1]{sponges} associated conditional probabilities defined by
\[
p (i_l \vert i_1, \ldots , i_{l-1})=
\begin{cases}
               \frac{\displaystyle\sum_{\substack{\textbf{j}=\left( j_1, \ldots, j_d\right)\in D \\ j_1=i_1, \ldots, j_{l-1}=i_{l-1}, j_l=i_l}}p_{\textbf{j}}}{\displaystyle\sum_{\substack{\textbf{j}=\left( j_1, \ldots, j_d\right)\in D \\ j_1=i_1, \ldots, j_{l-1}=i_{l-1}}}p_{\textbf{j}}} \hspace{25pt} \text{if } (i_1, \ldots, i_l) \in D_l   \\
              0 \hspace{125pt}  \text{if } (i_1, \ldots, i_l) \not\in D_l
            \end{cases}
\]
and it follows immediately from the definitions that $p(i_{l}\vert i_{1},\ldots,i_{l-1})=1/N(i_1,\ldots,i_{l-1})$ for $l=2,\ldots,d$ and $p(i_1\vert \emptyset)=1/N$ whenever $(i_1, \ldots, i_l)\in D$.  Roughly speaking, $p (i_l \vert i_1, \ldots , i_{l-1})$ is the probability of choosing $i_l$ as the next digit of $ ( i_1, \ldots , i_{l-1})$.

We recall that in the measure theoretic definition of Assouad dimension, our condition is dependent on the measure of balls which we would like to replace with approximate cubes. We will now prove that our approximate cubes can be used to find a suitable upper bound for the Assouad dimension. In Section \ref{measurethmproof} we will show that balls and approximate cubes are equivalent in our definition given the VSSC and in Section \ref{examplemeasures} we provide an example where the two have distinct properties.  The following proposition is inspired by \cite[Proposition 4]{bylund} and the measure theoretic definition of Assouad dimension.

\begin{prop}\label{adupmeasure}
Suppose there exists a Borel probability measure $\nu$ on $D^{\mathbb{N}}$ and constants $C >0$ and $s \geq 0$ such that for any $0< r<R\le1$ and $\omega \in D^{\mathbb{N}}$ we have
\[
\frac{\nu \left( Q(\omega,R)\right)}{\nu \left( Q(\omega,r)\right)}\le C \left( \frac{R}{r} \right)^s.
\]
Then $\dim_{\text{\emph{A}}} K \le s$.
\end{prop}

\begin{proof}

Given $0< r<R\le1$, choose an approximate cube $Q(\omega,R)$, and consider the approximate cubes of side length $r$ which are subsets of $Q(\omega,R)$. Let $N_{r,\omega}$ be the number of such approximate cubes (ones with side length $r$) and choose a representative set of centres $\left\{ \omega_i\right\}_{i=1, \ldots, N_{r,\omega}}$, one for each of the smaller approximate cubes.

For all $i=1,\ldots, N_{r,\omega}$, since $Q(\omega_i,R) =  Q(\omega,R)$, we have by our scaling assumption that
\[
\nu \left( Q (\omega,R) \right) \ = \  \nu \left( Q(\omega_i,R)\right) \ \le \  C \left(  \frac{R}{r} \right)^s  \nu \left( Q(\omega_i,r)\right).
\]

We also have
\[
\nu \left( Q (\omega,R) \right) \ \ge \ \tfrac{1}{2} \, \sum_{i=1}^{N_{r,\omega}} \nu \left( Q (\omega_i,r) \right) \ \ge \  \tfrac{1}{2} \,  N_{r,\omega} \, \min_{i=1, \ldots, N_{r,\omega}} \nu \left( Q (\omega_i,r) \right).
\]

The `1/2' in the above expression accounts for the fact that two approximate cubes may share a boundary, but in practise the boundaries carry zero measure so it is actually surplus to requirements. Therefore we obtain
\[
N_{r,\omega} \, \le \, 2 \, C \left( \frac{R}{r}\right)^s
\]
by picking $i$ in the first part to be the same as the one used for the minimum in the second part. We remember that in our first definition of Assouad dimension we were led to consider $N_{r}(B(x,R))$, which was the smallest number of balls needed to cover a larger ball whereas here we have $N_{r, \omega}$, which is the number of geometric approximate cubes contained in a larger geometric approximate cube. Thankfully these two quantities give us an equivalent definition for the Assouad dimension since any ball of radius $R$ centered in $K$ is contained in at most $3^d$ geometric approximate cubes of side length $R$ and similarly any geometric approximate cube of side length $r$ can be covered by at most $2\, n_1 \times \cdots \times n_d$ open sets of diameter $r$.  This yields
\[
\sup_{x\in K} N_r (B(x,R)\cap K) \leq 2\, n_1 \times \cdots \times n_d \, 3^{d} \, C\left(\frac{R}{r}\right)^{s}
\]
which in turn gives $\dim_\text{A} F \le s$, as required.
\end{proof}

We now wish to estimate the measure of approximate cubes, which is fortunately quite straightforward.  It follows from the definitions (and was observed by Olsen \cite[(6.2)]{sponges}) that
\begin{equation} \label{approxcubemeasure}
\tilde \mu(Q(\omega,r))=\prod^d_{l=1} \prod_{j=0}^{k_l(r)-1}p_l(\sigma^j\omega)
\end{equation}
where $p_l(\omega)=p(i_{1,l}\vert i_{1,1},\ldots,i_{1,l-1})$. We can use this to obtain our upper bound.

\begin{proof} By (\ref{approxcubemeasure}), we have
\begin{align*}
\frac{\tilde \mu(Q(\omega,R))}{\tilde \mu(Q(\omega,r))}
&=\frac{\prod^d_{l=1} \prod_{j=0}^{k_l(R)-1}p_l(\sigma^j\omega)}{\prod^d_{l=1} \prod_{j=0}^{k_l(r)-1}p_l(\sigma^j\omega)} \\
&=\prod^d_{l=1} \prod_{j=k_l(R)}^{k_l(r)-1}\frac{1}{p_l(\sigma^j\omega)} \\
&\leq \left( \prod_{j=k_1(R)}^{k_1(r)-1} N \right)\left(\prod^d_{l=2} \prod_{j=k_l(R)}^{k_l(r)-1} \max_{(i_1,\ldots,i_{l-1})\in D_{l-1}} N(i_1,\ldots,i_{l-1})\right) \\
&=N^{k_1(r)-k_1(R)}\left( \prod_{l=2}^d \max_{(i_1,\ldots,i_{l-1})\in D_{l-1}} N(i_1,\ldots,i_{l-1})^{k_l(r)-k_l(R)}\right) \\
& \leq N^{\log R/\log n_1-\log r/\log n_1+1}\left(\prod_{l=2}^d \max_{(i_1,\ldots,i_{l-1})\in D_{l-1}} N(i_1,\ldots,i_{l-1})^{\log R/\log n_l-\log r/\log n_l+1}\right) \\
&\le  N^{\log (R/r)/\log n_1} n_1\left(\prod_{l=2}^d \displaystyle\max_{(i_1,\ldots,i_{l-1})\in D_{l-1}}   N(i_1,\ldots,i_{l-1})^{\log (R/r)/\log n_l} \, n_l\right) \\
&= n_1 \times \cdots \times n_d\left(\frac{R}{r}\right)^{\displaystyle\frac{\log N}{\log n_1}}\left(\prod_{l=2}^d \left( \frac{R}{r}\right)^{\frac{\displaystyle\log \max_{(i_1,\ldots,i_{l-1})\in D_{l-1}} N(i_1, \ldots, i_{l-1})}{\displaystyle\log n_l}}\right) \\
&\le n_d^d\left( \frac{R}{r}\right)^{\displaystyle\frac{\log N}{\log n_1}+\sum_{l=2}^d \frac{\displaystyle\log \max_{(i_1,\ldots,i_{l-1})\in D_{l-1}} N(i_1, \ldots, i_{l-1})}{\displaystyle\log n_l}}.
\end{align*}

This estimate combined with Proposition \ref{adupmeasure} gives us the desired upper bound.
\end{proof}

\subsection{Lower bound for Assouad dimension} \label{lowerbound}

For the lower bound we will use `weak tangents', a technique due to Mackay and Tyson \cite[Proposition 6.1.5]{mackaytyson}.  The version we state and use here is a minor modification due to Fraser \cite[Proposition 7.7]{fraser}. This technique allows us to create simple tangent sets with the desired Assouad dimension and the following proposition gives us our lower bound.

\newpage

\begin{prop}[Very weak tangents]\label{tangents}
Let $X\subset \mathbb{R}^d$ be compact and let $F$ be a compact subset of $X$, Let $(T_k)$ be a sequence of bi-Lipschitz maps defined on $\mathbb{R}^d$ with Lipschitz constants $a_k, b_k \ge 1$ such that 
\[
a_k \lvert x-y \rvert \leq \lvert T_k(x) - T_k(y) \rvert \leq b_k \lvert x-y \rvert \,\,\,\,\,\,\,\, (x,y\in\mathbb{R}^d)
\]
and
\[
\sup_k b_k / a_k = C_0 <\infty
\]
and suppose that $T_k(F) \cap X \rightarrow \hat{F}$ in the Hausdorff metric. Then the set $\hat F$ is called a \emph{very weak tangent} to $F$ and, moreover,  $\dim_{\text{\emph{A}}} F \geq \dim_{\text{\emph{A}}} \hat{F}$.
\end{prop}

To simplify notation, we choose $\textbf{i}(l) = (i(l)_1, \ldots, i(l)_d)\in D$ for $l=2, \ldots d$ to be an element of $D$ which attains the maximum value for $N(i_1, \ldots, i_{l-1})$, i.e.
\[
N \left(i(l)_1, \ldots, i(l)_{l-1} \right)  \ =  \  \max_{(i_1,\ldots,i_{l-1})\in D_{l-1}} N(i_1, \ldots, i_{l-1}).
\]
There might be several possibilities for each $\textbf{i}(l)$, but thankfully it does not matter which one we pick. 

We will show that the set
\[
\hat{K}=\pi_1 K \times \prod_{l=2}^d K_l
\]
 is a subset of some very weak tangent to $K$, where $K_l$ is the Cantor set (or unit interval or single point) obtained by the IFS
\[
\left\{S_{l,1}(x)=\frac{x+j_{l,1}}{n_l} \, ,\ldots,\, S_{l, N(i_1, \ldots, i_{l-1})}(x)=\frac{x+j_{l,N(i_1, \ldots, i_{l-1})}}{n_l}\right\}
\]
acting on $[0,1]$ such that $j_{l,k}$ is the $l^{\text{th}}$ element of some $\textbf{j}_k \in D$ where for all $k=1, \ldots, N(i_1, \ldots, i_{l-1})$, the first $l-1$ coordinates of $\textbf{j}_k$ coincide with the first $l-1$ coordinates of $\textbf{i}(l)$. This IFS occurs naturally from the functions in the original IFS used to define our sponge: we simply take the $l^{\text{th}}$ component of the original function $S_{\textbf{j}_k}$ for each $k=1, \ldots, N(i_1, \ldots, i_{l-1})$. The set $\pi_1 K$ is the geometric projection of the sponge on to the first coordinate, or one can think of it as the Cantor set (or unit interval or single point) defined like our sponges but working in $[0,1]$ with defining set $\pi_1 D$.  In a slight abuse of notation, throughout the paper we use $\pi_1$ to denote projection onto the first coordinate in both the symbolic space $D$ and geometric space $\mathbb{R}^d$.

For $l=2, \dots, d$ and $m \in \mathbb{N}$, we let $K_l^m$ be the $m^{\text{th}}$ pre-fractal of $K_l$ where our initial set is $[0,1]$. In particular, the set $K_l^m$ is a union of $N(i_1, \ldots, i_{l-1})^m$ intervals of length $n_l^{-m}$.

Given a geometric approximate cube $\tau(Q) =\tau(Q(\omega,r))$, we define a bi-Lipschitz map $T^Q: \tau(Q)  \to [0,1]^d$ (or $T^Q:  \mathbb{R}^d  \to \mathbb{R}^d$) by
\[
T^Q(\textbf{x})= \begin{pmatrix}n_1^{k_1(r)}\left( x_1-\left( \frac{i_{1,1}}{n_1}+\ldots+\frac{i_{k_1(r),1}}{n_1^{k_1(r)}} \right) \right)\\\vdots\\n_d^{k_d(r)}\left( x_d-\left( \frac{i_{1,d}}{n_d}+\ldots+\frac{i_{k_d(r),d}}{n_d^{k_d(r)}} \right) \right)\end{pmatrix}.
\]
Thus $T^Q$ translates $\tau(Q)$ such that the point closest to the origin from the hypercuboid containing $\tau(Q)$ becomes the origin and then scales it up by a factor of $n_l^{k_l(r)}$ in each coordinate $l$. Thus these maps take the natural hypercuboid containing $\tau(Q)$ precisely to the unit cube $[0,1]^d$. We define $T_l^Q$ to be the $l^{\text{th}}$ component of $T^Q$, which is possible since $T^Q$ acts independently on each coordinate. These maps clearly satisfy the conditions imposed by Proposition \ref{tangents}, i.e., they are restrictions of bi-Lipschitz maps on $\mathbb{R}^d$ with constants $b_Q=\sup_{l=1, \ldots, d}n_l^{k_l(r)}$ and $a_Q=\inf_{l=1, \ldots, d}n_l^{k_l(r)}$ satisfying
\[
\frac{b_Q}{a_Q}\leq \sup_{l=1,\ldots,d}\frac{rn_l}{r}\leq n_d < \infty
\]
for any $Q$.   This follows from the definition of $k_l(r)$.

We define, for small $R$, $\omega(R)=\left(\textbf{i}_1, \textbf{i}_2, \ldots \right) \in D^{\mathbb{N}}$ where $\textbf{i}_t = (i_{t,1}, \dots, i_{t,d}) =\textbf{i}(l)$ for $t=k_l(R)+1, \ldots,k_{l-1}(R)$ for all $l=2, \ldots, d$. So $\omega(R)$ has the form
\[
\omega(R)=\left( \textbf{i}_1, \ldots, \underbrace{\textbf{i}(d), \ldots, \textbf{i}(d)}_{k_{d-1}(R)-k_{d}(R) \text{ times}}, \ \ \underbrace{\textbf{i}(d-1), \ldots, \textbf{i}(d-1)}_{k_{d-2}(R)-k_{d-1}(R) \text{ times}}, \ldots,  \underbrace{\textbf{i}(2), \ldots, \textbf{i}(2)}_{k_1(R)-k_{2}(R) \text{ times}}, \ldots  \right).
\]
Note that this step relies on $n_{l-1}<n_{l}$ for each $l=2, \ldots, d$, since otherwise $k_{l-1}(R)= k_{l}(R)$.  The idea of defining $\omega(R)$ like this, where one follows one word for a long time and then changes to another, is somewhat inspired by the approach of Fraser--Miao--Troscheit in the random setting \cite[Section 6.3.3]{fraserrandom}.

\begin{lma} \label{productform}
For $ R \in (0,1]$ small enough and $Q=Q(\omega(R),R)$, we have
\[
T^Q(\tau(Q)) \ \subseteq \  \pi_1 K \times \prod_{l=2}^d K_l^{k_{l-1}(R)-k_{l}(R)}.
\]
\end{lma}

\begin{proof}
We have 

\begin{align*}
Q\left( \omega \left( R\right), R \right)
&=\left\{ \omega'=\left( \textbf{j}_1, \textbf{j}_2 , \ldots \right)\in D^{\mathbb{N}} : \forall \, \,  l=1, \ldots, d \text{ and } \forall\, \, t= 1, \ldots, k_l(R) \text{ we have } j_{t,l}=i_{t,l} \right\} \\ \\
&\subseteq \left\{ \textbf{j} = (j_1, j_2 , \dots)   \in (\pi_1D)^{\mathbb{N}} :   \forall\, \, t= 1, \ldots, k_1(R) \text{ we have } j_{t}=i_{t,1} \right\} \\
&\, \qquad  \times \prod_{l=2}^d   \{ \textbf{j} = (j_1, j_2 , \dots)   \in \mathcal{I}_l^{\mathbb{N}} :  \forall\, \, t= 1, \ldots, k_l(R) \text{ we have } j_{t}=i_{t,l} \text{ and } \\
&\, \qquad \qquad \qquad  \forall\, \, t= k_l(R)+1, \ldots,k_{l-1}(R)  \text{ we have }(i(l)_1, \dots, i(l)_{l-1}, j_{t}) \in D_l \}.
\end{align*}
Observe that for each $l=1, \dots, d$,
\[
T^Q_l \circ \tau_l = \tau_l \circ \sigma^{k_l(R)}
\]
on $Q$, i.e., $T^Q_l$ acts symbolically by applying the left shift $k_l(R)$ times.  This yields

\newpage

\begin{align*}
T^Q(\tau(Q))
&\subseteq T^Q_1 \tau_1 \left\{ \textbf{j} = (j_1, j_2 , \dots)   \in (\pi_1D)^{\mathbb{N}} :   \forall\, \, t= 1, \ldots, k_1(R) \text{ we have } j_{t}=i_{t,1} \right\}\\
&\, \qquad  \times \prod_{l=2}^d  T^Q_l \tau_l  \{ \textbf{j} = (j_1, j_2 , \dots)   \in \mathcal{I}_l^{\mathbb{N}} :  \forall\, \, t= 1, \ldots, k_l(R) \text{ we have } j_{t}=i_{t,l} \text{ and }  \\
&\, \qquad \qquad \qquad  \forall \, \, t= k_l(R)+1, \ldots,k_{l-1}(R)  \text{ we have }(i(l)_1, \dots, i(l)_{l-1}, j_{t}) \in D_l \} \\
&=  \tau_1 \left\{ \textbf{j} = (j_1, j_2 , \dots)   \in (\pi_1D)^{\mathbb{N}} \right\} \\
&\, \qquad  \times \prod_{l=2}^d  \tau_l \,   \{ \textbf{j} = (j_1, j_2 , \dots)   \in \mathcal{I}_l^{\mathbb{N}} :  \forall \, \, t= 1, \ldots,k_{l-1}(R) -k_l(R) \\
&\, \qquad \qquad \qquad \qquad \qquad \qquad \qquad   \text{ we have }(i(l)_1, \dots, i(l)_{l-1}, j_{t}) \in D_l \} \\
&= \pi_1 K \times \prod_{l=2}^d K_l^{k_{l-1}(R)-k_{l}(R)}
\end{align*}
as required.
\end{proof}

The product set in Lemma \ref{productform} will act as an intermediary between $T^Q(\tau(Q(\omega(R),R))) $ and $\hat{K}$, in proving that $T^Q(\tau(Q(\omega(R),R))) \to \hat{K}$ in the Hausdorff metric as $R \to 0$.  For this next step we need to think of the geometry of the sets involved.

The product set
\[
\prod_{l=2}^d K_l^{k_{l-1}(R)-k_{l}(R)}
\]
has a natural decomposition into $(d-1)$-dimensional closed hypercuboids which are the products of basic intervals in the pre-fractal construction of each $K_l^{k_{l-1}(R)-k_{l}(R)}$.  In particular, the product decomposes as the union of
\[
M \  = \ \prod_{l=2}^d N(i_1, \dots i_{l-1})^{k_{l-1}(R)-k_{l}(R)}
\]
$(d-1)$-dimensional closed hypercuboids with side lengths $n_l^{k_{l-1}(R)-k_{l}(R)}$ for $l=2, \dots, d$. We note that this decomposition may not be a strict partition, but the interiors of the $(d-1)$-dimensional hypercuboids are pairwise disjoint.  Let this decomposition be labeled as $\{H_i \}_{i=1}^M$ and observe that
\[
\pi_1 K \times \prod_{l=2}^d K_l^{k_{l-1}(R)-k_{l}(R)} \ = \ \pi_1K \times \bigcup_{i=1}^M  H_i \ = \ \bigcup_{i=1}^M \pi_1K \times H_i
\]

Observe that by Lemma \ref{productform} we crucially have
\[
T^Q(\tau(Q)) \ = \  \bigcup_{i=1}^M \pi_1 K \times (H_i \cap T^Q(\tau(Q)))
\]
with each intersection $H_i \cap T^Q(\tau(Q))$ non-empty and, similarly, by the definition of $\hat{K}$
\[
\hat{K} \ = \  \bigcup_{i=1}^M \pi_1 K \times (H_i \cap \hat{K})
\]
with each intersection $H_i \cap \hat{K}$ non-empty.  Therefore, writing $\text{diam}( H_i)$ for the diameter of $H_i$, we have
\begin{align*}
d_{\mathcal{H}}\left( \hat{K},  \ T^Q(\tau(Q(\omega(R),R))) \right)
& \le \max_{i=1, \dots, M}  d_{\mathcal{H}}\left(  \pi_1 K \times (H_i \cap \hat{K}), \   \pi_1 K \times (H_i \cap T^Q(\tau(Q))) \right) \\  
& = \max_{i=1, \dots, M}  d_{\mathcal{H}}\left( H_i \cap \hat{K} , \ H_i \cap T^Q(\tau(Q)) \right) \\ 
& \leq \max_{i=1, \dots, M}  \text{diam}( H_i)  \\ 
& \le \sqrt{d} \max_{l=2, \ldots, d} n_l^{-\left( k_{l-1}(R)-k_{l}(R) \right)}\\
& \rightarrow 0
\end{align*}
as $R \to 0$, since $k_{l-1}(R)-k_{l}(R) \to \infty$.  This again relies on the $n_l$ being strictly increasing.  Consider the sequence of maps $T^Q$ for a sequence of approximate squares $Q=Q(\omega(R),R)$ with $R$ tending to zero.  Observe that we cannot quite conclude that $\hat{K}$ is a very weak tangent to $K$ because
\[
T^Q(K) \cap [0,1]^d \ \supseteq \ T^Q(\tau(Q(\omega(R),R))) \ \to \ \hat{K}
\]
and, strictly speaking, the containment may be strict.  This can happen if parts of neighbouring approximate cubes intersect with the natural hypercuboid containing $\tau(Q(\omega(R),R))$.  However, by a simple compactness and monotonicity argument this is not a problem.  The sequence $T^Q(K) \cap [0,1]^d $ is a sequence of non-empty compact subsets of $[0,1]^d $ and, since the space of non-empty compact subsets of $[0,1]^d$ is itself a compact metric space when equipped with the Hausdorff metric, we may extract a subsequence which converges to a non-empty compact set $E \subseteq [0,1]^d $ which is, by definition, a very weak tangent to $K$.  The containment outlined above implies that $\hat{K} \subseteq E$ by the following elementary lemma, the proof of which we leave to the reader.

\begin{lma}
Let $E_k, \,  F_k \subseteq [0,1]^d$ be  sequences of non-empty compact sets which converge in the Hausdorff metric to compact sets $E$ and $F$ respectively.   If $F_k \subseteq E_k$ for all $k$, then $F \subseteq E$.
\end{lma}

We are now ready to prove the lower bound for Assouad dimension.

\begin{proof} Standard results on the box dimensions of product sets \cite[Chapter 7]{falconer} and of self-similar sets \cite[Chapter 9]{falconer} imply that
\begin{eqnarray*}
\dim_{\text{B}} \hat{K} &=& \dim_{\text{B}} \pi_1K+\sum_{l=2}^d \dim_{\text{B}} K_l \\ \\&=& \frac{\log N}{\log n_1}+\sum_{l=2}^d \frac{\displaystyle\log\max_{(i_1,\ldots, i_{l-1})\in D_{l-1} } N(i_1, \ldots, i_{l-1})}{\log n_l}.
\end{eqnarray*}

Therefore using Proposition \ref{tangents} and monotonicity of Assouad dimension we obtain
\begin{eqnarray*}
\dim_{\text{A}} K  \  \ge \ \dim_{\text{A}} E  \  \ge \ \dim_{\text{A}} \hat{K} &  \ge &\dim_{\text{B}} \hat{K}  \\ 
 &=&  \frac{\log N}{\log n_1}+\sum_{l=2}^d \frac{\displaystyle\log\max_{(i_1,\ldots, i_{l-1})\in D_{l-1} } N(i_1, \ldots, i_{l-1})}{\log n_l}
\end{eqnarray*}
giving us our required lower bound.
\end{proof}

\subsection{Lower bound for lower dimension} \label{lowerlower}

For the lower bound we closely follow the method used for the upper bound of the Assouad dimension, although there are some notable differences in this case. We start by stating and proving a proposition similar to Proposition \ref{adupmeasure}, which is again inspired by \cite{bylund}.

\begin{prop}\label{lowlowmeasure}
Suppose there exists a Borel probability measure $\nu$ on $D^{\mathbb{N}}$ and constants $C >0$ and $s \geq 0$ such that for any $0< r<R\le1$ and $\omega \in D^{\mathbb{N}}$ we have
\[
\frac{\nu \left( Q(\omega,R)\right)}{\nu \left( Q(\omega,r)\right)}\ge C \left( \frac{R}{r} \right)^s.
\]
Then $\dim_{\text{\emph{L}}} K \ge s$.
\end{prop}

\begin{proof}

Given $0< r<R\le1$, we choose an approximate cube $Q(\omega,R)$ and as before we consider the approximate cubes of side length $r$ which are subsets of $Q(\omega,R)$. Let $N_{r,\omega}$ be the number of such approximate cubes (ones with side length $r$) and choose a representative set of centres $\left\{ \omega_i\right\}_{i=1, \ldots, N_{r,\omega}}$, one for each of the smaller approximate cubes.   For all $i=1,\ldots, N_{r,\omega}$ our scaling assumption yields
\[
\nu \left( Q (\omega,R) \right) \ = \ \nu \left( Q(\omega_i,R)\right)   \ \ge \  C \left(  \frac{R}{r} \right)^s  \nu \left( Q(\omega_i,r)\right).
\]
We also have
\[
\nu \left( Q (\omega,R) \right) \ \le \  \sum_{i=1}^{N_{r,\omega}} \nu \left( Q (\omega_i,r) \right)\ \le \  N_{r,\omega} \max_{i=1, \ldots, N_{r,\omega}} \nu \left( Q (\omega_i,r) \right).
\]
Therefore we obtain
\[
N_{r,\omega} \ge C \left( \frac{R}{r}\right)^s
\]
by picking $i$ in the first part to be the same as the one used for the maximum in the second part.  We have shown that every geometric approximate cube of side length $R$ must contain at least $C (R/r)^s$ geometric approximate cubes of side length $r$. Since the intersection of $K$ with any ball $B(x,R)$ centred in $K$ contains an approximate cube of side length $R/(n_d\sqrt{d})$, it must contain at least
\[
C \left( \frac{R/(n_d\sqrt{d})}{r }\right)^s
\]
approximate cubes of side length $r$ for any $r< R/(n_d\sqrt{d})$.  Any open set of diameter $r$ can intersect no more than $2^d$ of these approximate cubes of side length $r$ and so
\[
N_r (B(x,R)\cap K) \ \geq \  \frac{1}{2^d}\,  C\left( \frac{R/(n_d\sqrt{d})}{r }\right)^s  \ = \  \frac{1}{2^d \, n_d^s \, d^{s/2}} \, C\left(\frac{R}{r}\right)^{s}.
\]
To be precise, we should also deal with scales $r$ where $R/(n_d\sqrt{d}) \le r <  R $, but this range is trivial since
\[
N_r (B(x,R)\cap K) \ \geq \  1  \ \geq   \ \frac{1}{n_d^s \, d^{s/2}} \, \left(\frac{R}{r}\right)^{s}
\]
and so $\dim_\text{L} K \ge s$, as required.
\end{proof}

Now we can complete the proof of the lower bound by estimating the measure of approximate cubes as we did when finding the upper bound for Assouad dimension, again letting $\mu$ be the coordinate uniform measure.

\begin{proof} Using the formula for the measure of an approximate cube (\ref{approxcubemeasure}) we have
\begin{align*}
\frac{ \tilde \mu(Q(\omega,R))}{\tilde \mu(Q(\omega,r))}
&=\frac{\prod^d_{l=1} \prod_{j=0}^{k_l(R)-1}p_l(\sigma^j\omega)}{\prod^d_{l=1} \prod_{j=0}^{k_l(r)-1}p_l(\sigma^j\omega)} \\
&=\prod^d_{l=1} \prod_{j=k_l(R)}^{k_l(r)-1}\frac{1}{p_l(\sigma^j\omega)} \\
&\geq \left( \prod_{j=k_1(R)}^{k_1(r)-1} N \right)\left(\prod^d_{l=2} \prod_{j=k_l(R)}^{k_l(r)-1} \min_{(i_1,\ldots,i_{l-1})\in D_{l-1}} N(i_1,\ldots,i_{l-1})\right) \\
&=N^{k_1(r)-k_1(R)}\left( \prod_{l=2}^d \min_{(i_1,\ldots,i_{l-1})\in D_{l-1}} N(i_1,\ldots,i_{l-1})^{k_l(r)-k_l(R)}\right) \\
& \geq N^{\log R/\log n_1-\log r/\log n_1-1}\left(\prod_{l=2}^d \min_{(i_1,\ldots,i_{l-1})\in D_{l-1}} N(i_1,\ldots,i_{l-1})^{\log R/\log n_l-\log r/\log n_l-1}\right) \\
&\ge N^{\log(R/r)/\log n_1}n_1^{-1}\left(\prod_{l=2}^d \displaystyle\min_{(i_1,\ldots,i_{l-1})\in D_{l-1}} N(i_1,\ldots,i_{l-1})^{\log (R/r)/\log n_l} \, n_l^{-1}\right) \\
& = (n_1 \times \cdots \times n_d)^{-1}\left(\frac{R}{r}\right)^{\displaystyle\frac{\log N}{\log n_1}}\left(\prod_{l=2}^d \left( \frac{R}{r}\right)^{\frac{\displaystyle\log \min_{(i_1,\ldots,i_{l-1})\in D_{l-1}} N(i_1, \ldots, i_{l-1})}{\displaystyle\log n_l}}\right) \\
&\ge  n_d^{-d}\left( \frac{R}{r}\right)^{\displaystyle\frac{\log N}{\log n_1}+\sum_{l=2}^d \frac{\displaystyle\log \min_{(i_1,\ldots,i_{l-1})\in D_{l-1}} N(i_1, \ldots, i_{l-1})}{\displaystyle\log n_l}}.
\end{align*}

This estimate combined with Proposition \ref{lowlowmeasure} gives us the required lower bound.
\end{proof}

\subsection{Upper bound for lower dimension} \label{upperlower}

The upper bound for lower dimension is similar to the lower bound for Assouad dimension given in Section \ref{lowerbound}.  However, there is a complication due to the fact that the monotonicity argument given at the end of the proof does not apply.  This is for two reasons: the lower dimension is not monotone, and the inclusion between the `good' product set $\hat{K}$ and the genuine very weak tangent $E$ goes the wrong way for estimating dimension.

We begin by building a very weak tangent to $K$, just as we did in Section \ref{lowerbound}.  Choose $\textbf{i}(l)= (i(l)_1, \ldots, i(l)_d)\in D$ for $l=2, \ldots , d$ to be an element of $D$ which gives us the \emph{minimum} value for $N(i_1, \ldots, i_{l-1})$ and once again if there are multiple choices for such minimising elements, then choose one arbitrarily.  Let $\hat{K}$ be defined as the following product of sets:
\[
\hat{K}=\pi_1 K \times \prod_{l=2}^d K_l
\]
where $K_l$ is the Cantor set (or unit interval or single point) obtained by the IFS 
\[
\left\{S_{l,1}(x)=\frac{x+j_{l,1}}{n_l} \, ,\ldots ,\, S_{ l, N(i_1, \ldots, i_{l-1})}(x)=\frac{x+j_{l,N(i_1, \ldots, i_{l-1})}}{n_l}\right\}
\]
acting on $[0,1]$ such that $j_{l,k}$ is the $l^{\text{th}}$ element of some $\textbf{j}_k \in D$ where for all $k=1, \ldots, N(i_1, \ldots, i_{l-1})$, the first $l-1$ coordinates of $\textbf{j}_k$ coincide with the first $l-1$ coordinates of $\textbf{i}(l)$.  For the moment we assume that $\hat{K} \cap (0,1)^d \neq \emptyset$.

It can be shown using the same argument as in Section \ref{lowerbound} that $K$ has a very weak tangent $E$ which contains $\hat{K}$ as a subset.  Closer inspection of the proof in Section \ref{lowerbound} reveals that points in $E$ which are not in $\hat{K}$ must lie on the boundary of the unit cube.  We record this important fact for later.  Again, it follows from standard results on the box dimensions of product sets \cite[Chapter 7]{falconer} and of self-similar sets \cite[Chapter 9]{falconer}  that
\[
\dim_{\text{B}} \hat{K} \ = \ \dim_{\text{B}} \pi_1K+\sum_{l=2}^d \dim_{\text{B}} K_l \  = \ \frac{\log N}{\log n_1}+\sum_{l=2}^d \frac{\displaystyle\log\min_{(i_1,\ldots, i_{l-1})\in D_{l-1} } N(i_1, \ldots, i_{l-1})}{\log n_l}.
\]

Unfortunately, it is easy to construct examples of sets $F$ with very weak tangents $\hat{F}$ which satisfy $\dim_{\text{L}} F > \dim_{\text{L}} \hat{F}$; the opposite of what we want.  The reason for this is that sets with isolated points have lower dimension zero and it is possible for a very weak tangent to a set to have isolated points even if the original set did not.  However, it turns out that (with mild additional assumptions) $\dim_{\text{L}} F \leq \underline{\dim}_{\text{B}} \hat{F}$, which is sufficient in our setting.

\begin{prop}\label{tangents2}
Let $F\subseteq [0,1]^d$ be compact and suppose that $\hat F$ is a very weak tangent to $F$ in the sense of Proposition \ref{tangents} (with $X=[0,1]^d$). If $\hat{F} \cap (0,1)^d \neq \emptyset$, then $\dim_{\text{\emph{L}}} F \leq \underline{\dim}_{\text{\emph{B}}} \hat{F}$.
\end{prop}

\begin{proof}
	Let $\alpha< \dim_{\text{L}} F$ be arbitrary.  Let $\hat{y} \in \hat{F} \cap (0,1)^d$ and fix a constant $t>0$ such that $B(\hat{y},2t) \subseteq (0,1)^d$.  Let $r >0$ be small and let $\{U_i\}_i$ be an $r$-cover of $\hat{F}$ by open balls.   Let $k$ be sufficiently large to guarantee that $d_\mathcal{H}(\hat{F}, T_k(F) \cap  [0,1]^d) < \min\{r,t\}$. It follows that there exists $y \in T_k(F) \cap  [0,1]^d$ such that $B(y,t) \subseteq (0,1)^d$ and also that, writing $U_i'$ for the open ball centered at the same point as $U_i$ but with twice the radius,  $\{U'_i\}_i$  is a $2r$-cover of $T_k(F) \cap  [0,1]^d$.  The bi-Lipschitz condition on $T_k$ implies that
\[
B(T_k^{-1}(y) , b_k^{-1} t) \cap F \subseteq T_k^{-1} ( B(y,t) ) \cap F \subseteq T_k^{-1}( [0,1]^d) \cap F = T_k^{-1}( T_k(F) \cap  [0,1]^d)
\]
and that $\{T_k^{-1}(U'_i)\}_i$ is a $2a_k^{-1}r$-cover of every set in the above chain, in particular $B(T_k^{-1}(y) , b_k^{-1} t) \cap F$.    Therefore by the definition of lower dimension the number of sets in the original cover cannot be fewer than
\[
C \, \left( \frac{b_k^{-1}t}{2a_k^{-1}r} \right)^{\alpha} \ \geq  \  C \,  (2C_0)^{-\alpha} t^{\alpha} \, r^{-\alpha}
\]
where $C$ is a constant depending only on $\alpha$ coming straight from the definition of lower dimension.   This demonstrates that $\underline{\dim}_{\text{B}} \hat{F} \geq \alpha$, which proves the result since $\alpha< \dim_{\text{L}} F$ was arbitrary.
\end{proof}

Finally, we have to deal with the fact that $\hat{K}$ is a subset of a very weak tangent $E$ and so the dimension estimate goes the wrong way.  This is easily handled by the following simple lemma.

\begin{lma} \label{subtangent}
Suppose $\hat{F}$ is a very weak tangent to a non-empty compact set $F \subseteq [0,1]^d$ in the sense of Proposition \ref{tangents} (with $X = [0,1]^d$).  Let $X' \subseteq [0,1]^d$ be a closed hypercube which is the image of $[0,1]^d$ under a similarity $S$ and such that $X' \cap \hat{F} \neq \emptyset$.  Then $S^{-1} (X' \cap \hat{F} )$ is also a very weak tangent to $F$.
\end{lma}

\begin{proof}
This follows immediately from the definition using the sequence of maps $S \circ T_k$, where $T_k$ is the sequence of bi-Lipschitz maps used in demonstrating that $\hat{F}$ is a very weak tangent.
\end{proof}

We can now complete the proof.   Let $X' \subseteq (0,1)^d$ be a closed hypercube which is the image of $[0,1]^d$ under a similarity $S$ and such that $\underline{\dim}_{\text{B}} S^{-1} (X' \cap E) \  = \ \dim_{\text{B}} \hat{K}$ and $S^{-1} (X' \cap E) \cap (0,1)^d \neq \emptyset$.  We can do this since $E \cap (0,1)^d \  = \  \hat{K} \cap (0,1)^d \neq \emptyset$.  Then it follows from Lemma \ref{subtangent} that $S^{-1} (X' \cap E )$ is a very weak tangent to $K$ and then Proposition \ref{tangents2} yields
\[
\dim_{\text{L}} K \  \le \ \underline{\dim}_{\text{B}} S^{-1} (X' \cap E ) \  = \ \dim_{\text{B}} \hat{K}  \ = \  \frac{\log N}{\log n_1}+\sum_{l=2}^d \frac{\displaystyle\log\min_{(i_1,\ldots, i_{l-1})\in D_{l-1} } N(i_1, \ldots, i_{l-1})}{\log n_l}
\]
as required.

It only remains to deal with the case where $\hat{K} \cap (0,1)^d =  \emptyset$.  This can be dealt with via a simple approximation argument and we just sketch the ideas, leaving the details to the interested reader.  Consider the iterated sets $D^k$ for large $k \in \mathbb{N}$ and associate the natural $k$-fold iteration of the original IFS, which has the same attractor, $K$.  Given any $\varepsilon>0$, one may choose $k$ sufficiently large such that there exists a choice of $\textbf{i}(l)'= (i(l)'_1, \ldots, i(l)'_d)\in D^k$ for $l=2, \ldots , d$ such that the tangent constructed by the method outlined above will have box dimension no more than $\varepsilon$ larger than the target upper bound and also intersects the interior of the unit hypercube.  One then follows the above proof and then lets $\varepsilon$ tend to zero.  Note that the ability to do this is reliant on our assumption that $K$ does not lie in a hyperplane, but we recall that this assumption was made without loss of generality.

One possible way of choosing the $\textbf{i}(l)'$ would be to first choose $\textbf{i}(l) \in D$ as before, i.e., to yield a very weak tangent with optimal dimension, $L$, (but which may lie on the boundary of the unit hypercube) and also choose $\textbf{j} \in D^d$ which corresponds to a hypercuboid which does not touch the boundary of the unit cube (such a point $\textbf{j} \in D^d$ must exist or else $K$ lies in a hyperplane).  Then choose $\textbf{i}(l)' = \textbf{j} \textbf{i}(l)^{k-d}$ for large $k>d$.  A short and pleasant calculation yields that, for example, 
\[
N(i(l)'_1, \ldots, i(l)'_{l-1}) \, \leq \,  N(i(l)_1, \ldots, i(l)_{l-1})^{k-d}  n_l^d
\]
where we abuse notation slightly by also using $N$ for the analogous function for the iterated IFS.  These choices of $\textbf{i}(l)'$ clearly yield a very weak tangent which does not lie on the boundary of $[0,1]^d$ and one may bound its box dimension by

\newpage 

\begin{eqnarray*}
&\,& \hspace{-30mm} \frac{\log N^k}{\log n_1^k} \, + \, \sum_{l=2}^d \frac{\displaystyle\log N(i(l)'_1, \ldots, i(l)'_{l-1})}{\log n_l^k} \\ \\
& \leq& \frac{\log N}{\log n_1} \, + \, \sum_{l=2}^d \left( \frac{\displaystyle (k-d) \log  N(i(l)_1, \ldots, i(l)_{l-1})}{k \log n_l} \, + \, \frac{d}{k} \right) \\ \\
&\leq & L + \frac{d(d-1)}{k}  \\ \\
&\leq & L + \varepsilon 
\end{eqnarray*}
provided $k \geq d(d-1)/\varepsilon$.

\subsection{Sharpness of the coordinate uniform measure} \label{measurethmproof}

This proof follows easily from \cite[Proposition 6.2.1]{sponges}, stated below in our notation.

\begin{prop}[Olsen]
Let $\omega \in D^{\mathbb{N}}$ and $k \in \mathbb{N}$.
\begin{enumerate} 
\item If the VSSC is satisfied, then $B\left( \tau(\omega), 2^{-1}n_1^{-k}\right)\cap K \subseteq \tau \left(Q\left( \omega, n_1^{-k} \right) \right).$ 
\item  $\tau \left(Q\left( \omega, 1/n_1^k \right) \right) \subseteq B\left( \tau(\omega), (n_1+\dots+n_d)n_1^{-k}\right).$ 
\end{enumerate}
\end{prop}

Throughout this section we assume the VSSC is satisfied which, among other things, implies that $\tau$ is injective and that $\mu(\tau(Q(\omega,r))) = \tilde \mu (Q(\omega, r))$.  In Sections \ref{upperbound} and \ref{lowerlower} we have already proved that $\mu$ is `sharp' for approximate squares in that for all $\omega\in D^{\mathbb{N}}$ and $0<r<R \leq 1$ we have
\[
n_d^{-d} \left(\frac{R}{r}\right)^{\dim_\text{L}K} \  \leq \ \frac{\mu(\tau (Q(\omega,R)))}{\mu(\tau (Q(\omega,r)))}\  \leq \  n_d^d \left(\frac{R}{r}\right)^{\dim_\text{A}K}.
\]
Fix $x \in K$ and let $\omega\in D^{\mathbb{N}}$ be such that $\tau(\omega) = x$.  Also fix $0<r<R \leq 2^{-1} n_1^{-1}$ and let
\[
R' \ = \ \inf_{ k \in \mathbb{N}}\left\{ 2^{-1}n_1^{-k} \ : \ 2^{-1}n_1^{-k} \geq R \right\}
\]
and
\[
r'  \ = \  \sup_{ k \in \mathbb{N}} \left\{  (n_1+\dots+n_d) n_1^{-k} \ : \  (n_1+\dots+n_d)n_1^{-k} \leq r \right\}
\]
and observe that $R \leq R' < n_1R$ and $n_1^{-1} r < r' \leq r$.  The reason for these definitions is that we may apply Olsen's proposition to the scales $r'$ and $R'$.  In particular we have
\begin{eqnarray*}
\frac{\mu(B(x,R))}{\mu(B(x,r))} \  \le \ \frac{\mu(B(x,R'))}{\mu(B(x,r'))}  &\le& \frac{\mu(\tau(Q(\omega,2R')))}{\mu\left(\tau \left(Q\left(\omega,r'/(n_1+\dots+n_d)\right)\right)\right)} \\ \\
&\leq & n_d^d\left(\frac{2(n_1+\dots+n_d)R'}{r'}\right)^{\dim_\text{A}K} \\ \\
&\leq & n_d^d\, (2(n_1+\dots+n_d) n_1^{2})^{\dim_\text{A}K}\left(\frac{R}{r}\right)^{\dim_\text{A}K}
\end{eqnarray*}
which is sufficient, setting $C_1 =  n_d^d\, (2(n_1+\dots+n_d) n_1^{2})^{\dim_\text{A}K}$.  The lower scaling estimate concerning lower dimension is proved similarly and omitted.

\section{Further discussion and examples} \label{qanda}

In this section we will provide examples to illustrate several interesting aspects of the dimension theory of sponges and we will discuss our results along the way. We will finish with some open questions.

\subsection{Worked example}\label{example}

Here we consider a basic example of a Bedford-McMullen sponge $K$ where $d=3$ as this is easier to visualise. Let $n_1=2$, $n_2=3$, $n_3=4$ and consider the sponge $K$ generated by
\[
D=\left\{ (0,0,0), (0,0,3), (0,1,2), (1,0,2), (1,1,0), (1,1,1), (1,1,2), (1,2,0), (1,2,2), (1,2,3) \right\}.
\]

For this sponge we have $N=2$, 
\[
\begin{aligned}[c]
\max N(i_1)&=3\\
\max N(i_1, i_2)&=3\\
\end{aligned}
\qquad \qquad
\begin{aligned}[c]
\min N(i_1)&=2\\
\min N(i_1, i_2)&=1.\\
\end{aligned}
\]

So, by Theorem \ref{mainthm},
\[
\dim_{\text{A}} K = \frac{\log 2}{\log 2}+ \frac{\log 3}{\log 3}+\frac{\log 3}{\log 4}=2+ \frac{\log 3}{\log 4} \approx 2.792
\]
 and
\[
\dim_\text{L} K = \frac{\log 2}{\log 2}+ \frac{\log 2}{\log 3}+\frac{\log 1}{\log 4}=1+\frac{\log 2}{\log 3} \approx 1.631.
\]
We can also use the formulae for the box and Hausdorff dimensions (discussed in Section \ref{evenmorequestions}), proved by Kenyon and Peres \cite{kenyonperes}, to get
\[
\dim_\text{B} K=\frac{\log 2}{\log 2}+\frac{\log \left(  5/2 \right)}{\log3}+\frac{\log \left( 10/5 \right)}{\log 4}=1+\frac{\log \left(  5/2 \right)}{\log3}+\frac{\log 2}{\log 4}\approx 2.3340
\]
and
\[
\dim_\text{H} K= \log_2 [(2^{\log3/\log 4}+1)^{\log 2/\log 3}+(2\times 3^{\log 3 /\log 4}+1)^{\log 2/\log 3}]\approx 2.296.
\]
We will compare these formulae in Section \ref{evenmorequestions}, for now we simply remark that $\dim_\text{L} K < \dim_\text{H} K < \dim_\text{B} K < \dim_\text{A} K$ holds for this example.

The coordinate uniform measure for this example is given by weights:
\[
\begin{aligned}[c]
p_{(0,0,0)}&=\frac{1}{2 \times 2 \times 2}=\frac{1}{8}\\
p_{(0,0,3)}&=\frac{1}{2 \times  2\times 2}=\frac{1}{8}\\
p_{(0,1,2)}&=\frac{1}{2 \times  2\times 1}=\frac{1}{4}\\
p_{(1,0,2)}&=\frac{1}{2 \times  3\times 1}=\frac{1}{6}\\
p_{(1,1,0)}&=\frac{1}{2 \times  3\times 3}=\frac{1}{18}
\end{aligned}
\qquad \qquad
\begin{aligned}[c]
p_{(1,1,1)}&=\frac{1}{2 \times  3\times 3}=\frac{1}{18}\\
p_{(1,1,2)}&=\frac{1}{2 \times  3\times 3}=\frac{1}{18}\\
p_{(1,2,0)}&=\frac{1}{2 \times  3\times 3}=\frac{1}{18}\\
p_{(1,2,2)}&=\frac{1}{2 \times  3\times 3}=\frac{1}{18}\\
p_{(1,2,3)}&=\frac{1}{2 \times  3\times 3}=\frac{1}{18},
\end{aligned}
\]
and the $\textbf{i}(l)$ used in the proof of the lower bound  for the Assouad dimension can be, for example,
\[
\textbf{i}(2)=(1,1,1)
\]
\[
\textbf{i}(3)=(1,2,2).
\]

Therefore $\hat{K}=[0,1] \times [0,1] \times K_3$ where $K_3$ is the Cantor set defined by the IFS
\[
\left\{ S_{3,1}(x)=\frac{x}{4} , \ S_{3,2}(x)=\frac{x+2}{4} , \ S_{3,3}(x)=\frac{x+3}{4} \right\}.
\]
We recall that $\hat{K}$ is not necessarily a very weak tangent to $K$ but is a subset of a very weak tangent, see the discussion towards the end of Section \ref{lowerbound}. One can find similar examples for the lower dimension, but we leave this to the reader.

\subsection{A self-affine carpet with no doubling Bernoulli measures} \label{examplemeasures}

In this section we provide a simple, self-contained example of a self-affine carpet which does not carry a doubling Bernoulli measure.  Such examples were given previously by Li, Wei and Wen \cite{doublingcarpets}, but we include an example here for completeness.

Let $d=2$, $n_1 = 2$, $n_2=4$ and
\[
D = \{(0,1), (1,1), (1,3)\}.
\]
For a given probability vector $\{p_{(0,1)}, p_{(1,1)}, p_{(1,3)}\}$, let $\mu$ be the associated Bernoulli measure supported on $K$ and $\tilde{\mu}$ the Bernoulli measure on $D^\mathbb{N}$.  Observe that each of the three probabilities must be strictly positive in order to fulfill the requirement that the support of $\mu$ is $K$.  We claim that there is no choice of probabilities for which $\mu$ is doubling.

\begin{figure}[H]
	\centering
		\includegraphics[width = 70mm]{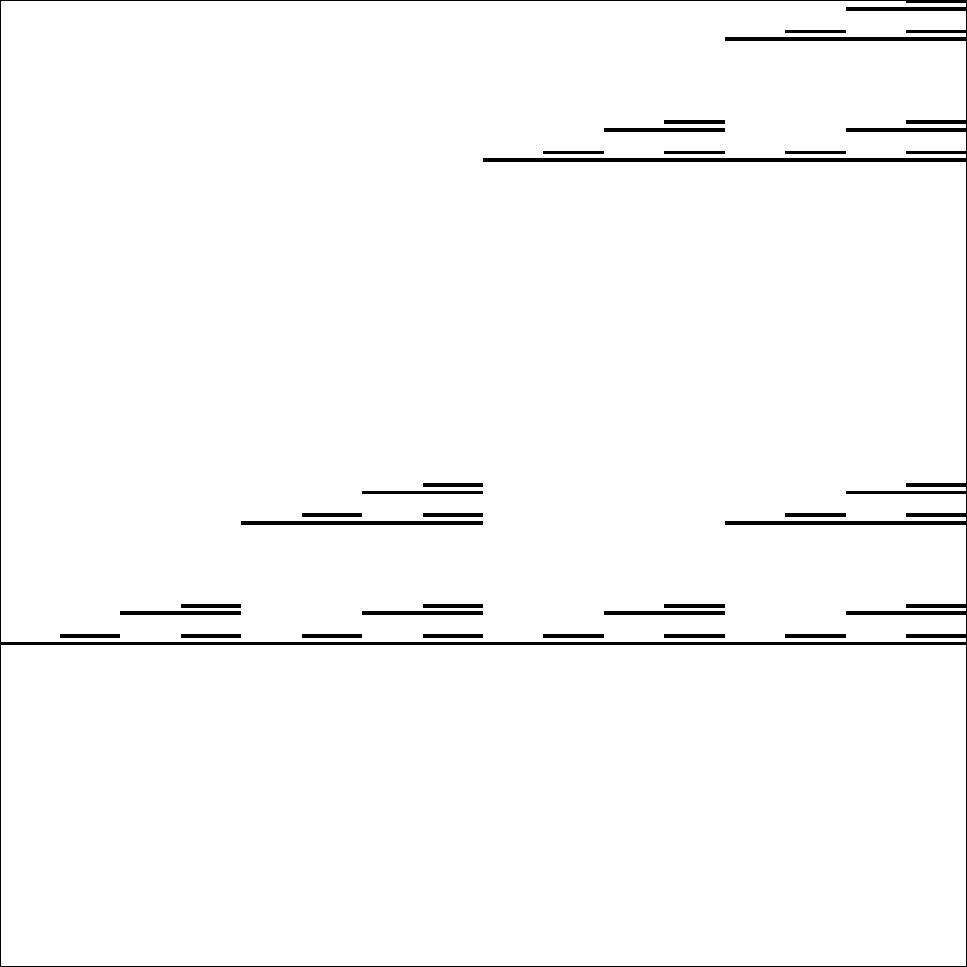}
	\caption[Carpet with no doubling Bernoulli measures]{A self-affine carpet with no doubling Bernoulli measures.}
	\label{fig:nobernoulli}
\end{figure}

Fix a strictly positive probability vector.  Observe that it is sufficient to find a sequence of triples $(\omega, \omega',R) \in D^\mathbb{N} \times D^\mathbb{N} \times (0,1]$ such that $\tau(Q(\omega,R))$ and $\tau(Q(\omega',R))$ do not coincide, but are `neighbours' in in that the natural hypercuboids containing them intersect along an edge, and such that
\[
\frac{\tilde{\mu}(Q(\omega,R))}{\tilde{\mu}(Q(\omega',R))}
\]
 does not remain bounded away from zero and infinity.  To see why this is sufficient observe that for such a triple it is possible to place a point $x \in \tau(Q(\omega,R))$ such that $B(x,R/5) \cap K \subseteq \tau(Q(\omega,R))$  (and with $B(x,R/5)$ lying inside the interior of the natural hypercuboid containing $\tau(Q(\omega,R))$ so that this is the only contribution to its measure) and $B(x,9R) \supseteq \tau(Q(\omega,R)) \cup \tau(Q(\omega',R))$ and similarly with the roles of $\omega$ and $\omega'$ changed.

For large $k \in \mathbb{N}$, let $R=4^{-k}$ and observe that $k_1(R) = 2k$ and $k_2(R) = k$.  First consider
\[
\omega=\left(  \underbrace{(0,1), \ldots, (0,1)}_{k \text{ times}}, (0,1), \underbrace{(1,1), \ldots, (1,1)}_{k -1\text{ times}}, \ldots  \right)
\]
and
\[
\omega'=\left(  \underbrace{(0,1), \ldots, (0,1)}_{k \text{ times}}, (1,1), \underbrace{(0,1), \ldots, (0,1)}_{k-1 \text{ times}}, \ldots  \right).
\]
It is easy to see that the right hand edge of $\tau(Q(\omega,R))$ and the left hand edge of $\tau(Q(\omega',R))$ are in common.  Moreover, using (\ref{approxcubemeasure}) we see that
\[
\frac{\tilde{\mu}(Q(\omega,R))}{\tilde{\mu}(Q(\omega',R))} \ = \  \frac{p_{(0,1)}^{k+1} (p_{(1,1)}+p_{(1,3)})^{k-1}}{p_{(0,1)}^{2k-1} (p_{(1,1)}+p_{(1,3)}) } \ = \  \left(\frac{p_{(1,1)}+p_{(1,3)} }{p_{(0,1)}}\right)^{k-2}.
\]
For this expression to remain bounded away from zero and infinity for all $k$, we must have $p_{(0,1)} = p_{(1,1)}+p_{(1,3)}$, which we will assume from now on.  Now consider the words 

 \[
\omega=\left(  (0,1), \underbrace{(1,1), \ldots, (1,1)}_{2k-1 \text{ times}}, \ldots  \right)
\]
and
\[
\omega'=\left(  (1,1), \underbrace{(0,1), \ldots, (0,1)}_{2k-1 \text{ times}}, \ldots  \right).
\]
Again, it is easy to see that the right hand edge of $\tau(Q(\omega,R))$ and the left hand edge of $\tau(Q(\omega',R))$ are in common.  Moreover, using (\ref{approxcubemeasure}) we see that
\[
\frac{\tilde{\mu}(Q(\omega,R))}{\tilde{\mu}(Q(\omega',R))} \ = \  \frac{p_{(0,1)} (p_{(1,1)}+p_{(1,3)})^{2k-1} \left( \frac{p_{(1,1)}}{p_{(1,1)}+p_{(1,3)}} \right)^{k-1}}{(p_{(1,1)}+p_{(1,3)})p_{(0,1)}^{2k-1}  \left( \frac{p_{(1,1)}}{p_{(1,1)}+p_{(1,3)}} \right)} \ = \  \left(\frac{p_{(1,1)}}{p_{(1,1)}+p_{(1,3)} }\right)^{k-2}.
\]
However, for this expression to remain bounded away from zero and infinity for all $k$, we must have $p_{(1,1)} = p_{(1,1)}+p_{(1,3)}$, which forces $p_{(1,3)} = 0$ which is forbidden.  This completes the proof.

\subsection{A self-affine sponge without strictly increasing $(n_l)$} \label{examplenonstrict}

Now we will calculate the Assouad, lower, Hausdorff and box dimensions of a Bedford-McMullen sponge in $\mathbb{R}^3$ where $n_2=n_3$. We cannot use our formulae for such a set so we will have to use different techniques to find the Assouad and lower dimensions.

Let $n_1=3$ and $n_2=n_3=4$ and consider the sponge $K$ defined by
\[
D=\left\{ (0,0,0), (0,3,0), (0,3,1), (0,3,2), (0,3,3), (2,0,0), (2,3,0), (2,3,1), (2,3,2), (2,3,3)  \right\}.
\]
For this example the sponge can be viewed as the product of the middle 3rd Cantor set $E$ (in the first coordinate) and a self-similar set $F$ generated by 5 maps with similarity ratio $1/4$ (in second and third coordinates). The following figure can help illustrate this.

\begin{figure}[H]
	\centering
		\includegraphics{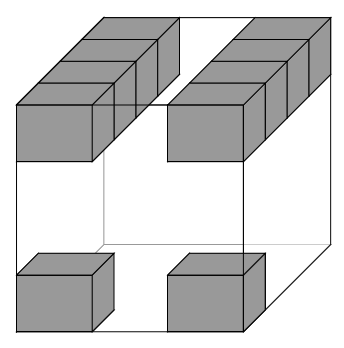}
	\caption[Example sponge]{A sponge for which our formulae do not hold.}
	\label{fig:bad sponge example}
\end{figure}

Since we can write this set as a product of two self-similar sets the dimension of the sponge is equal to the sum of the dimensions of the product sets \cite[Theorem 2.1]{fraser}, i.e.
\[
\dim_\text{A} K= \dim_\text{L} K= \dim_\H K=\dim_\text{B} K= \dim_\text{B} E + \dim_\text{B} F.
\]

So for this example the actual dimensions are
\[
\dim_\text{A} K= \dim_\text{L} K=\frac{\log 2}{\log 3}+\frac{\log 5}{\log 4} \approx 1.792,
\]
but using the formulae we obtain in this paper we would get a value of 
\[ \frac{\log 2}{\log 3}+\frac{\log 2}{\log 4}+\frac{\log 4}{\log 4}\approx 2.131,\] 
for the Assouad dimension and 
\[\frac{\log 2}{\log 3}+\frac{\log 2}{\log 4}+\frac{\log 1}{\log 4}\approx 1.131\]
for the lower dimension, which are incorrect.  Interestingly, the dimension formulae for Hausdorff and box dimension do remain valid even in the case when $n_l=n_{l+1}$ for some $l$, see \cite[Corollary 4.1.2]{sponges}.

This example demonstrates a type of `discontinuity' present for the Assouad and lower dimensions, which is not present for the Hausdorff and box dimensions.  Since the parameter space is discrete for our class of self-affine sponges it is a little misleading to talk about discontinuity, but we observe that a genuine discontinuity in the same spirit is already apparent in the planar case.  We will illustrate this by example.  Fix $\lambda \in (0,1/2]$ and consider the IFS acting on the square consisting of the three maps
\[
T_1(x,y) = (x/2,\lambda y) , \quad T_2(x,y) = (x/2+1/2,\lambda y) , \quad \text{and} \quad T_3(x,y) = (x/2+1/2,\lambda y+1-\lambda).
\]
For $\lambda \in (0,1/2)$, the attractor of this IFS is a self-affine carpet of the type considered by Lalley and Gatzouras \cite{lalley-gatz} and for $\lambda=1/2$ it is the self-similar Sierpi\'nski triangle. To emphasise the dependence on $\lambda$ we denote the attractor by $F_\lambda$.
\begin{figure}[H]
	\centering
		\includegraphics[width=160mm]{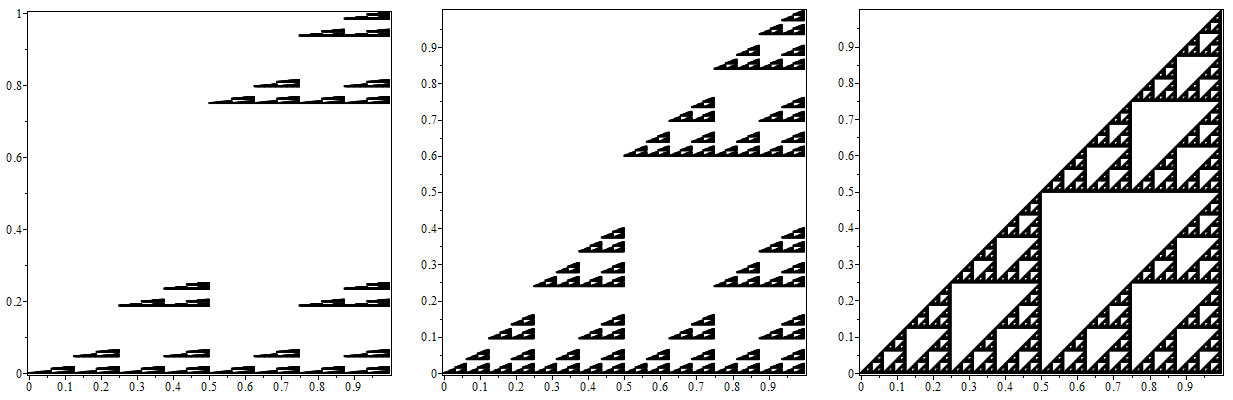}
	\caption[Discontinuity in dimensions]{Plots of $F_\lambda$ for $\lambda$ equal to $1/4$, $2/5$ and $1/2$ respectively.}
	\label{fig:discontinuity}
\end{figure}
Using previously known formulae for the box, Hausdorff, Assouad, and lower dimensions, we obtain the following: for $\lambda \in (0,1/2)$ we have
\[
\dim_{\text{L}} F_\lambda \ = \  1, \qquad \qquad  \dim_\text{H} F_\lambda  \ = \  \frac{\log \left(1+2^{-\log 2/\log \lambda}\right)}{\log 2}
\]
\[
\dim_\text{B} F_\lambda \ = \ 1+\frac{\log (3/2)}{-\log \lambda}, \qquad \qquad \dim_\text{A} F_\lambda \ = \  1+\frac{\log 2}{-\log \lambda}
\]
using results from \cite{lalley-gatz} and \cite{mackay, fraser}; and for $\lambda=1/2$ we have
\[
\dim_\text{A} F_\lambda \ = \  \dim_\text{L} F_\lambda \ = \ \dim_\H F_\lambda \ = \ \dim_\text{B} F_\lambda \ = \  \frac{\log 3}{\log 2}
\]
using the self-similarity of $F_{1/2}$.  One then observes that the function $\lambda \mapsto \dim F_\lambda$ is continuous for the box and Hausdorff dimensions but discontinuous at $\lambda = 1/2$ for Assouad and lower dimensions.  As far as we know this type of discontinuity has not been observed before, although the dimension formulae used above were previously known.
\begin{figure}[H]
	\centering
		\includegraphics[width=90mm]{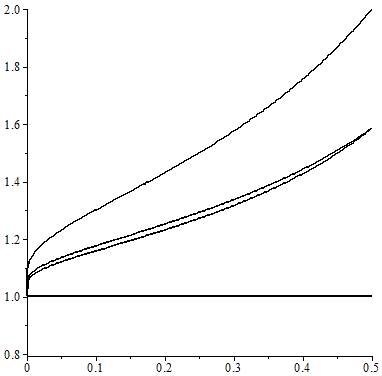}
	\caption[Discontinuity in dimensions]{A plot of $\dim_{\text{L}} F_\lambda$, $\dim_{\text{H}} F_\lambda$, $\dim_{\text{B}} F_\lambda$ and $\dim_{\text{A}} F_\lambda$ for $\lambda \in (0,1/2]$.}
	\label{fig:dimensions}
\end{figure}

\subsection{Extension of Mackay's  dichotomy for dimension} \label{evenmorequestions}

Mackay \cite{mackay} observed a pleasant dichotomy for the dimensions of the planar self-affine carpets.  In particular, he noted that either the Assouad, box and Hausdorff dimensions are all equal or all distinct.  This dichotomy was extended to include the lower dimension by Fraser \cite{fraser}.  In this section we show that this dichotomy extends to sponges.

We say a sponge has \emph{uniform fibres} if for all $l=1, \dots, d-1$ we have
\[
N(i_1, \dots, i_l)  =  N(j_1, \dots, j_l) 
\]
for all $(i_1, \dots, i_l) ,(j_1, \dots, j_l) \in D_l$.  It follows immediately from Theorem \ref{mainthm} that if a sponge has uniform fibres, then the lower, Hausdorff, box and Assouad dimensions are all equal.  Kenyon and Peres \cite[Proposition 1.3 (ii)]{kenyonperes} showed that the box and Hausdorff dimensions coincide if and only if the sponge has uniform fibres.  This means that to prove the extension of Mackay's dichotomy, we must only show that if the fibres are not uniform, then the box and Assouad dimensions are distinct and the lower and Hausdorff dimensions are distinct.

Kenyon's and Peres' \cite{kenyonperes} formula for the  box dimension of a sponge is
\[
\dim_{\text{B}} K \ = \ \frac{\log N}{\log n_1} \ + \  \sum_{l=2}^d \frac{1}{\log n_l} \log \left( \frac{\lvert D_l \rvert}{\lvert D_{l-1} \rvert} \right).
\]
Since, for any $l=2, \dots, d$, we have
\[
\lvert D_l \rvert = \sum_{(i_1, \ldots, i_{l-1})\in D_{l-1}} N(i_1, \ldots, i_{l-1})
\]
we observe that
\[
\frac{\lvert D_l \rvert }{\lvert D_{l-1} \rvert } \ \leq \  \max_{(i_1, \ldots, i_{l-1})\in D_{l-1}} N(i_1, \ldots, i_{l-1}).
\]
Comparing the formulae for the box dimension and the Assouad dimension we see that for the two to be equal we must have equality above for all $l=2, \dots, d$, which can only happen in the uniform fibres case.

The Hausdorff--lower dimension case is slightly more involved, simply due to the more complicated formula for Hausdorff dimension which we take from \cite{kenyonperes}.  This is defined inductively as follows.  For all $(i_1, \dots, i_d) \in D$, let $Z_d(i_1, \dots, i_d) = 1$.  For $(i_1, \dots, i_{l-1}) \in D_{l-1}$, assuming we have defined  $Z_l(i_1, \dots, i_l)$ for all $(i_1, \dots, i_l) \in D_l$ and setting $n_{d+1}=n_d$, we let
\[
Z_{l-1}(i_1, \dots, i_{l-1}) \ = \ \sum_{i_l: (i_1, \dots, i_l) \in D_l} Z_l(i_1, \dots, i_l)^{\log n_l/\log n_{l+1}}
\]
Finally
\[
\dim_\text{H} K \ = \ \frac{\log Z_0}{\log n_{1}}.
\]
In order to compare the Hausdorff dimension and the lower dimension, we define analogous quantities $Z'_l(i_1, \dots, i_l)$ inductively by first setting $Z'_d(i_1, \dots, i_d) = 1$ and then, for $(i_1, \dots, i_{l-1}) \in D_{l-1}$, assuming we have defined  $Z'_l(j_1, \dots, j_l)$ for all $(j_1, \dots, j_l) \in D_l$, by letting
\[
Z'_{l-1}(i_1, \dots, i_{l-1}) \ = \ N(i_1, \dots, i_{l-1}) \min_{(j_1, \dots, j_l) \in D_l} Z'_l(j_1, \dots, j_l)^{\log n_l/\log n_{l+1}}.
\]
It is easily checked that for all $l =1, \dots, d$ and $(i_1, \dots, i_l) \in D_l$ one has $Z_l(i_1, \dots, i_l) \geq Z'_l(i_1, \dots, i_l)$ and, moreover, that
\[
\dim_\text{L} K \ = \ \frac{\log Z'_0}{\log n_{1}}.
\]
Now suppose $K$ does not have uniform fibres.  Then there exists a uniquely defined $t \in \{1, \dots, d-1\}$ such that $N(i_1, \dots, i_t)$ is not constant in $(i_1, \dots, i_t) \in D_t$ but $N(i_1, \dots, i_l)$ \emph{is} constant in $(i_1, \dots, i_l) \in D_l$ for all $l > t$ with $l<d$.  It follows from the above definitions that
\[
Z_l(i_1, \dots, i_l)= Z'_l(i_1, \dots, i_l)
\]
and is constant in $(i_1, \dots, i_l) \in D_l$ provided $l>t$.  We also have
\[
Z_t(i_1, \dots, i_t)= Z'_t(i_1, \dots, i_t)
\]
but this is not constant in $(i_1, \dots, i_t) \in D_t$.  This guarantees that
\[
Z_{t-1}(i_1, \dots, i_{t-1})> Z'_{t-1}(i_1, \dots, i_{t-1})
\]
for some $(i_1, \dots, i_{t-1}) \in D_{t-1}$.  Since the $Z_l(i_1, \dots, i_l)$ are \emph{strictly increasing} as functions of the corresponding $Z_{l+1}(i_1, \dots, i_{l+1})$, this guarantees that $Z_0 > Z_0'$ and therefore that the Hausdorff and lower dimensions are distinct.

\subsection{Open questions} \label{morequestions}

To conclude we will recapitulate some of the open questions that we find particularly interesting. Theorem \ref{measurethm} demonstrates that there are sharp measures supported on sponges in the VSSC case.  Finding sharp measures in the general case is a more complicated problem because the simple class of Bernoulli measures is of no use, see Section \ref{examplemeasures} and \cite{doublingcarpets}.

\begin{ques} Can one always find sharp measures for sponges (or carpets), even when the VSSC is not satisfied? \end{ques}

It is also curious at first sight that our formulae for the Assouad and lower dimensions are only valid when the $n_l$ are strictly increasing, but that the formulae for box and Hausdorff dimensions do remain valid in this case.  The non-strictly increasing case has potential for further investigation.

\begin{ques} What are the Assouad and lower dimensions of a sponge when $n_l=n_{l+1}$ for some $l$? \end{ques}

It is also very natural to consider more general self-affine sets.  We see no significant roadblock in extending our arguments to the natural higher dimensional analogue of Lalley--Gatzouras carpets, assuming the appropriate version of strictly increasing $n_l$.  The extension of the Bara\'nski class could be rather more difficult and we believe there will be several different cases to consider.  Also the carpets of Feng--Wang and Fraser would be interesting to study but as discussed in \cite{fraser} we need to develop new tools which do not rely on approximate squares or cubes in this setting.  Indeed, the Assouad and lower dimensions are not known for these carpets even in the planar case.  We also point out that the box and Hausdorff dimensions are not known in the higher dimensional setting beyond the sponge case considered in this paper and in \cite{kenyonperes, sponges}, so there is plenty to think about here.

\begin{ques} What are the dimensions of the natural higher dimensional analogues of the more general carpets considered by Lalley--Gatzouras \emph{\cite{lalley-gatz}}, Bara\'nski  \emph{\cite{baranski}}, Feng--Wang \emph{\cite{fengaffine}} and Fraser \emph{\cite{fraser_box}} ?
\end{ques}

\begin{centering}

\textbf{Acknowledgements}

\end{centering}

We thank an anonymous referee for making several very useful and detailed comments on the paper, leading to significant improvements in the precision and exposition of our work.

\end{document}